\documentclass[10pt,reqno]{amsart}

\usepackage{amssymb}
\usepackage{amscd}
\usepackage{amsfonts}
\usepackage{setspace}
\usepackage{version}
\usepackage{amsmath}
\usepackage{amsthm}

\theoremstyle{definition}
\newtheorem*{theorem}{Theorem}

\def\R{\mathbf{R}}

\renewcommand{\leq}{\leqslant}

\begin{document}

\title{Almost all sets of $d+2$ points on the $(d-1)$-sphere are not subtransitive}



\author{Sean Eberhard}
\address{Centre for Mathematical Sciences, Wilberforce Road, Cambridge CB3 0WA}
\email{s.eberhard@dpmms.cam.ac.uk}

\onehalfspace

\begin{abstract}
We generalise an argument of Leader, Russell, and Walters to show that almost all sets of $d + 2$ points on the $(d - 1)$-sphere $S^{d-1}$ are not contained in a transitive set in some $\R^n$.
\end{abstract}
\maketitle

A finite subset of $\R^d$ is called \emph{transitive} if it has a transitive group of symmetries, and \emph{subtransitive} if it is a subset of a transitive set in some $\R^n$, where possibly $n>d$. Clearly every subtransitive set lies on a sphere. The converse was answered negatively by Leader, Russell, and Walters~\cite{LRWcyclicquad,LRWeuclidramsey} in connection to some conjectures in Euclidean Ramsey theory. Their key idea in~\cite{LRWcyclicquad} was to show more strongly that almost all cyclic quadrilaterals are not even \emph{affinely subtransitive}; that is, they do not embed into a transitive set even by a (nonconstant) affine map $x\mapsto Ax + b$.

The purpose of this note is to point out that both the result and the argument in \cite{LRWcyclicquad} generalise straightforwardly: almost all sets of $d+2$ points on the $(d-1)$-sphere are not affinely subtransitive. On the other hand, it is not hard to see that every affinely independent set of $d+1$ points, in other words a nondegenerate simplex, is subtransitive (see, e.g., \cite{FRsimplices} or \cite{LRWeuclidramsey}), so this result is best possible.

\begin{theorem}
Almost every set of $d+2$ points on the $(d-1)$-sphere $S^{d-1}\subset\R^d$ is not affinely subtransitive.
\end{theorem}
\begin{proof}
Let $x_0,\ldots,x_{d+1}$ be chosen uniformly at random from $S^{d-1}$. Since there are only countably many finite groups, each of which has only countably many orthogonal representations up to orthogonal conjugacy, it suffices to fix a finite subgroup $G$ of $O(n)$, and elements $g_1,\ldots,g_{d+1}\in G$, and show that almost surely there is no nonconstant affine $f:\R^d\to\R^n$ such that
\begin{equation}\label{subtrans}
f(x_k) = g_k f(x_0) \quad\text{for all $k = 1,\ldots,d + 1$}.
\end{equation}

If $x_0,\ldots,x_d$ are affinely independent then they may be affinely mapped to the standard affine basis $0, e_1,\ldots,e_d$ of $\R^d$. The image $\alpha = (\alpha_1,\ldots,\alpha_d)\in\R^d$ of $x_{d+1}$ is then uniquely determined. The
function $\phi : (x_0,\ldots, x_{d+1}) \mapsto \alpha$ thus defined is a rational map, and moreover the image of $\phi$ is ``large"\footnote{For instance, $\text{image}(\phi)\supset (0,1)^{d-1}\times(1,\infty)$. In the language of algebraic geometry, $\phi$ is a \emph{dominant} rational map: its image is not contained in any proper subvariety.}.

If $\phi(x_0,\ldots,x_{d+1})=\alpha$, then \eqref{subtrans} has a nonconstant affine solution if and only if it does when $(x_0,\ldots,x_{d+1})$ is replaced by $(0, e_1,\ldots,e_d,\alpha)$; call this condition \eqref{subtrans}$^\prime$. Writing $f(x) = Ax + b$ with $A\in\R^{n\times d}$ and $b\in \R^n$, the conditions $k = 1, \ldots, d$ of \eqref{subtrans}$^\prime$ are equivalent to
\[A = (g_1 b-b,\ldots, g_d b - b),\]
while the final condition states that
\begin{equation}\label{bcondition}
\sum_{k=1}^d \alpha_k(g_k b-b) + b = g_{d+1} b.
\end{equation}

Note that \eqref{bcondition} is an $n\times n$ linear condition on $b$. If the only solutions are fixed points of $\{g_1,\ldots,g_{d+1}\}$, then $A$ must be $0$, so $f$ must be constant. Quotienting by the subspace of fixed points of $\{g_1,\ldots,g_{d+1}\}$, we therefore obtain an $n\times n^\prime$ (where $n^\prime \leq n$) linear system which has a nonzero solution if and only if \eqref{subtrans}$^\prime$ has a nonconstant affine solution. Since each $n^\prime\times n^\prime$ minor of this system is a polynomial in $\alpha_1,\ldots,\alpha_d$, it follows, unless each such polynomial is identically zero, that the set of $\alpha$ such that \eqref{subtrans}$^\prime$ has a nonconstant affine solution is contained in a proper subvariety of $\R^d$. Since $\phi$ has large image, it then follows that the set of $(x_0,\ldots,x_{d+1})$ such that \eqref{subtrans} has a nonconstant affine solution is contained in a proper subvariety of $(S^{d-1})^{d+2}$.

Thus it remains only to show that some $n^\prime\times n^\prime$ minor of this system is not identically zero, or equivalently that \eqref{subtrans}$^\prime$ does not always have a nonconstant affine solution. That is, we must rule out the possibility that $\{0, e_1,\ldots, e_d,\alpha\}$ is affinely subtransitive for \emph{every} $\alpha\in\R^d$. But if $\alpha = (1/(2d),\ldots,1/(2d))$ then $\{0,e_1,\ldots,e_d,\alpha\}$ is not even convex, so it cannot even be mapped onto a sphere with a nonconstant affine map.
\end{proof}

For other results about subtransitive sets, see~\cite{J}.

I am grateful to Imre Leader for his comments and discussion.

\bibliography{subtrans}{}
\bibliographystyle{abbrv}
\end{document}